\documentclass[a4paper, 12pt,oneside,reqno]{amsart}
\usepackage[a4paper]{geometry}
\usepackage{amssymb,amsmath,amsthm}
\usepackage[T1]{fontenc}
\usepackage[breaklinks=true]{hyperref}

\DeclareSymbolFont{symbolsC}{U}{pxsyc}{m}{n}
\SetSymbolFont{symbolsC}{bold}{U}{pxsyc}{bx}{n}
\DeclareFontSubstitution{U}{pxsyc}{m}{n}
\DeclareMathSymbol{\medcirc}{\mathbin}{symbolsC}{7}

\DeclareMathSymbol{\medbullet}{\mathbin}{symbolsC}{8}

\usepackage{xcolor}

\usepackage{bbm}
\usepackage{graphicx}

\usepackage{tikz}
\usetikzlibrary{tikzmark}

\def\bl{\mathrm{bl}}

\def\Var{\mathrm{Var}}
\def\I{\mathrm{I}}
\def\dim{\mathrm{dim}}

\def\Nov{\mathrm{Nov}}
\def\Com{\mathrm{Com}}
\def\Adm{\mathrm{Adm}}
\def\T{\mathrm{T}}
\def\As{\mathrm{As}}

\def\NAP{\mathrm{NAP}}
\def\Lie{\mathrm{Lie}}
\def\Perm{\mathrm{Perm}}

\def\Zinb{\mathrm{Zinb}}
\def\Leib{\mathrm{Leib}}

\def\pre{\mathrm{pre}}

\def\di{\mathrm{di}}

\def\Perm{\mathrm{Perm}}

\def\Alt{\mathrm{Alt}}

\def\LSym{\mathrm{LSym}}

\let\<\langle
\let\>\rangle

\newtheorem{definition}{Definition}
\newtheorem{lemma}{Lemma}
\newtheorem{proposition}{Proposition}
\newtheorem{theorem}{Theorem}
\newtheorem{corollary}{Corollary}
\newtheorem{remark}{Remark}
\newtheorem{example}{Example}

\title{Initial pre-algebras as a generalization of dendriform algebras}

\author{P. S. Kolesnikov}

\address{Sobolev Institute of Mathematics, Novosibirsk, Russia}

\email{pavelsk77@gmail.com}

\author{B. K. Sartayev$^{*}$}

\address{Narxoz University, Almaty, Kazakhstan and SDU University, Kaskelen, Kazakhstan}

\email{baurjai@gmail.com}



\thanks{${}^{*}$Corresponding author: Bauyrzhan Sartayev   (baurjai@gmail.com)}

\begin{document}

\begin{abstract}
We continue the study of \emph{initial dialgebras} defined in~\cite{DMS2026}.
For a binary operad $\Var$ we define 
the class of initial pre-$\Var$-algebras and the corresponding operad $\pre\Var^{\I}$
in such a way that 
\[
(\di\Var^{\I})^{!}=(\pre(\Var^{!}))^{\I}
\]
in the case when $\Var $ is quadratic.
We propose an intuitive algorithm for finding 
the defining relations of the operad $\pre\Var^{\I}$ in the case when $\Var$ is a binary quadratic operad.
We also study free initial pre-algebras in the associative and commutative settings. For the nonsymmetric operad $\pre\As^{\I}$, we construct a Gröbner--Shirshov basis in the free magma operad, describe a linear basis in terms of admissible decorated planar binary trees, and establish a bijection between these trees and certain combinatorial objects.
\end{abstract}

\maketitle

\section{Introduction}

The theory of \emph{dialgebras} and their \emph{pre}-analogues originates from the problem of constructing
nonsymmetric and skew-symmetric counterparts of classical algebraic structures.
In the 1990s, Loday introduced \emph{associative dialgebras}, that is, algebras with two binary products
satisfying additional axioms, as a natural framework for Leibniz algebras--non-skew-symmetric analogues of
Lie algebras~\cite{Loday,LodPir,LV}. Any associative dialgebra produces a Leibniz algebra by combining its two products,
in the same way as an associative algebra produces a Lie algebra via commutators.
From the operadic viewpoint, the operad of associative dialgebras is Koszul dual to the operad of
\emph{dendriform algebras}~\cite{LV}. Dendriform algebras, also introduced by Loday, are equipped with two
operations whose sum is associative; equivalently, they provide a \emph{pre-associative} splitting of the
associative law. This reflects a general principle: many quadratic varieties admit dual ``di'' and ``pre''
versions. For instance, \emph{pre-Lie} (left-symmetric) algebras are characterized by the property that the
commutator bracket is Lie, whereas \emph{Zinbiel} algebras (``Leibniz'' spelled backwards) are those whose
symmetrized product is associative and commutative. Thus, dialgebraic and pre-algebraic structures are
naturally organized by operads, Koszul duality, and Manin products~\cite{LV}. In particular, for a quadratic
operad $\Var$ one has
\[
(\di\Var)^{(!)}=(\Var\circ\Perm)^{(!)}=\Var^{(!)}\bullet\pre\Lie=\pre(\Var^{(!)}).
\]

For a binary quadratic operad $\Var$, the operad $\di\Var$ is obtained by replication of
the product into two operations and imposing certain relations on them; at the level of polynomial identities,
this is closely related to KP-type algorithms (see \cite{Bremner2012}).
Dually, the splitting procedure leads to \emph{pre}-algebras that form a class $\pre\Var$, where the defining operations
refine the multiplication in $\Var$ so that an appropriate (anti)symmetrization recovers an algebra in $\Var$.
The basic examples are classical: the dendriform (pre-associative) and Zinbiel (pre-commutative) operads are
Koszul dual to the diassociative and Leibniz operads, respectively~\cite{LV}.
In the general operadic setting, analogous splitting procedures producing pre-versions of operads were
developed in~\cite{BBGN13,GubKol-13}.

Recent results illustrate both the importance and the subtlety of this direction.
For example, in the vertex-algebraic setting the problem of preservation of the Dong property under replication and splitting
constructions exhibits delicate behaviour and yields counterexamples in the pre-Lie and dendriform cases~\cite{DongProp}. On the dialgebraic side, new embedding theorems and structural results continue to
appear; in particular, Novikov dialgebras and their relations to Perm algebras with derivations were studied
in~\cite{perm5,SartDzhMashNovikov}. Further results on pre-Novikov type structures were obtained in
\cite{GaoGuoHanZhang,KolMashSar,LiHongQFNov}. Various dialgebraic structures on the polynomial algebra
$\Bbbk[x]\otimes\Bbbk[x]$ were investigated in~\cite{4guys}, and new classes of dendriform-type algebras and
related constructions were considered in~\cite{dend1,dend2}.

In~\cite{DMS2026}, a class of \emph{initial dialgebras} was introduced, and a procedure was proposed for
deriving defining identities of an initial $\Var$-dialgebra from the defining identities of $\Var$.
The main motivation for introducing initial dialgebras is to obtain a necessary and sufficient condition
ensuring that every $\Var$-dialgebra satisfies the defining identity of $\Var$ with respect to the product
\[
a*b:=a\vdash b+a\dashv b.
\]
These developments indicate the necessity of an explicit operadic framework describing how to pass from $\Var$
to its initial dialgebraic and pre-analogues, and how these objects interact via Koszul duality.

The aim of the present paper is to advance this program in the initial associative and initial pre-Lie
dialgebraic settings, and to develop a general procedure producing \emph{initial} pre-$\Var$ operads for a
given binary operad $\Var$. Our approach combines explicit operadic computations, Koszul duality, and
rewriting techniques in the magma operad. The main results are as follows.

\medskip
\noindent\textbf{Operadic description of initial dialgebras.}
Given a binary operad $\Var$, we note that the class $\di\Var^{\I}$ of initial di-$\Var$-algebras is governed by an operad whose defining relations join those of $\di\Var = \Perm\medcirc \Var$ and $\Var^*:=\Lie\Adm\medbullet \Var$, 
where $\Perm$ and $\Lie\Adm $ are the well-known operads of left commutative associative and non-associative Lie-admissible algebras, respectively, $\medcirc $ and $\medbullet $
stand for the Manin white and black products, respectively.

\medskip
\noindent\textbf{A general definition of initial pre-algebras.}
Based on the duality between $\di\Var$ and $\pre\Var^!$ in the quadratic case, we define 
the operad
$\pre\Var^{\I}$ as Koszul dual to $\di(\Var^!)^{\I}$. This definition is then generalized to an arbitrary binary (not necessarily quadratic) operad $\Var$ assuming 
the desired relations form the intersection of the T-ideals of the varieties $\pre\Var$
and $\T\Perm\medcirc \Var$, where $\T\Perm=\Lie\Adm^!$.


\medskip
\noindent\textbf{A rewriting algorithm for quadratic identities.}
We propose an explicit algorithm to find the defining identities of the operad $\pre\Var^{\I}$ for a binary quadratic operad $\Var$. The algorithm does not require 
complete computation of the Koszul dual operad but still needs some linear algebra analysis of the space of quadratic defining relations of $\Var$.

\medskip
\noindent\textbf{Gr\"obner--Shirshov bases and combinatorics.}
For the nonsymmetric operad $\pre\text{-}\As^{\I}$, we construct a Gröbner--Shirshov basis in the free magma operad in terms of an explicit system of rewriting rules. This yields a linear basis consisting of admissible decorated planar binary trees. We further establish a bijection between these trees and lattice paths of type $F$, thereby proving that
\[
\dim(\pre\textrm{-}\As^\I(n))=A106228(n-1),n\geq 1,
\]
where $A106228$ is the corresponding OEIS sequence; see \cite{comb1,comb2,HuhKimSeoShin2024}. For the operad $\pre\text{-}\Com^{\I}$, an explicit description of a linear basis remains open. Nevertheless, computations in low degrees using the computer algebra package \cite{Albert} yield the sequence
\[
1,2,7,32,181,1232,9787,88832,907081,10291712,
\]
which coincides up to degree $10$ with the OEIS sequence A006154; see \cite{comb3,comb4}.

\medskip
Since varieties defined by polynomial identities of degrees $2$ and $3$ are in one-to-one correspondence with
quadratic operads, we use the same terminology for a variety $\Var$ and for the corresponding operad
throughout the paper. We denote by $\Com$ the variety of associative-commutative algebras. All algebras are
considered over a field $\Bbbk$ of characteristic $0$.

\section{Dialgebras and initial dialgebras}

In this section, we recall 
the definition of an initial $\di\Var $-algebra and derive the  defining relations
of Koszul dual operads to 
$\di\As$ and $\di\Lie$, where $\As$ and $\Lie$ correspond to associative and Lie algebras, respectively.

Suppose $\Var $ is a variety of algebras 
with a family of binary operations $\mu_i$, $i\in I$, satisfying some multi-linear identities. The corresponding operad (see, e.g., \cite{BremDots}) is also denoted $\Var $.

In particular, the variety $\Var=\Perm$
has one operation $\mu(x,y)=xy$ which is associative and left-commutative:
\[
(xy)z=x(yz), \quad x(yz)=y(xz).
\]

For every $P\in \Perm$ and $A\in \Var $
define the structure of an 
algebra with doubled family of operations $\mu_i^{1}, \mu_i^{2}$, $i\in I$,
in the following way:
\[
\mu_i^1(p\otimes a, q\otimes b)
  =qp\otimes \mu_i(a,b),
  \quad 
\mu_i^2(p\otimes a, q\otimes b)
  =pq\otimes \mu_i(a,b),
\]
for $a,b\in A$, $p,q\in P$.
The set of all identities that hold on all 
such algebras $P\otimes A$ (with all possible $P\in \Perm$ and $A\in \Var$) 
defines a variety of $\di\Var$-algebras.
The corresponding operad 
$\di\Var$ coincides with the Manin white product of operads $\Perm\medcirc \Var$, 
the latter coincides with their Hadamard product $\Perm\otimes \Var$ (see \cite{KolVar, Vallette2008}).

An algorithm to deduce the defining relations of the operad $\di\Var$
(i.e., the identities that define the class of $\di\Var$-algebras) works as follows.
First, we need {\em zero-identities} that do not depend on $\Var$:
\begin{equation}\label{eq:Zero-identities}
    \mu_i^2(\mu_j^1(x,y),z) = \mu_i^2(\mu_j^2(x,y),z) ,
    \quad 
    \mu_i^1(x,\mu_j^2(y,z)) = 
    \mu_i^1(x,\mu_j^1(y,z)),
\end{equation}
for all $i,j\in I$.
Next, 
for every multi-linear defining identity $f(x_1,\dots, x_n)=0$  of $\Var$ and for every $k=1,\dots,n$ consider the identity
\[
f^k(x_1,\dots, x_n) := f(x_1,\dots, \dot x_k,\dots , x_n)=0,
\]
where the terms $f^k$ are defined by an 
apparent inductive principle. 
Namely, if $f = \mu_i(g,h)$, where $x_k$ appears in $g$ and $x_l$ is any variable in $h$ then 
\[
f^k = \mu_i^1(g^k,h^l)
\]
(due to \eqref{eq:Zero-identities}, it does not depend on the choice of $l$).
Similarly, if $x_k$ appears in $h$ and $x_l$ is a variable in $g$ then 
\[
f^k = \mu_i^2(g^l, h^k).
\]
Then $\di\Var $ is defined by the set of identities \eqref{eq:Zero-identities} along with all $f^k(x_1,\dots, x_n)$, $k=1,\dots, n$, where $f$ ranges the set of defining identities of $\Var$.

\begin{definition}[\cite{DMS2026}]\label{defn:di-initial}
A $\di\Var$-algebra $A$ with operations 
$\mu_i^1$, $\mu_i^2$, $i\in I$, 
is said to be {\em initial}
if the operations 
\[
\mu_i^* = \mu_i^1 + \mu_i^2,\quad i\in I,
\]
satisfy the identities of $\Var $.
\end{definition}

The class of all initial $\di\Var$-algebras 
is denoted $\di\Var^{\I}$.

\begin{example}
If $\Var = \Lie$ is the variety of Lie algebras with one skew-symmetric operation 
$\mu(x,y) = [x,y]$ satisfying the Jacobi identity then $\di\Lie^{\I}$
consists of Lie-admissible Leibniz algebras.
\end{example}

\begin{example}
An algebra with two operations
$\vdash $ and $\dashv $ is an initial $\di\As$-algebra if and only if it satisfies the following identities:
\begin{equation}\label{eq:Zero-dash}
(x_1\dashv x_2)\vdash x_3=(x_1\vdash x_2)\vdash x_3,
\quad
x_1\dashv(x_2\dashv x_3)=x_1\dashv(x_2\vdash x_3),
\end{equation}
\[
(x_1\dashv x_2)\dashv x_3=x_1\dashv(x_2\dashv x_3),
\quad 
(x_1\vdash x_2)\vdash x_3= x_1\vdash(x_2\vdash x_3),
\]
\[
(x_1\vdash x_2)\dashv x_3=x_1\vdash(x_2\dashv x_3)
\]
and
\begin{equation}\label{as1}
(x_1\vdash x_2)\vdash x_3=x_1\dashv (x_2\dashv x_3).    
\end{equation}
\end{example}

Hence, 
\eqref{as1} is the only identity 
we need to distinguish $\di\As$ and $\di\As^{\I}$.

\begin{proposition}[\cite{DMS2026}]\label{prop:diAS-I!}
The Koszul dual operad
$(\di\As^{\I})^!$ is generated by 
two operations $\prec $ and $\succ $
satisfying the following relations:
\begin{multline}\label{dend1}
(x_{1} \succ x_{2}) \succ x_{3} + (x_{1} \prec x_{2}) \succ x_{3}
- x_{1} \succ (x_{2} \succ x_{3})  \\
\qquad + (x_{1} \prec x_{2}) \prec x_{3}
- x_{1} \prec (x_{2} \succ x_{3})
- x_{1} \prec (x_{2} \prec x_{3})=0
\end{multline}
and
\begin{equation}\label{dend2}
x_1\succ(x_2\prec x_3) = (x_1\succ x_2)\prec x_3 .
\end{equation}
\end{proposition}

\begin{proof}
This is a straightforward computation based 
on the following criterion \cite{GK94}
(see also \cite{DongProp}):
the dual operations $\prec $ and $\succ $
satisfy exactly the identities  needed for the skew-symmetric bracket
\begin{multline*}
[y_1\otimes x_1,\, y_2\otimes x_2] \\
= (y_1\succ y_2)\otimes(x_1\vdash x_2)
-(y_2\succ y_1)\otimes(x_2\vdash x_1)
+(y_1\prec y_2)\otimes(x_1\dashv x_2)
-(y_2\prec y_1)\otimes(x_2\dashv x_1)
\end{multline*}
to satisfy the Jacobi identity
(for $x_i$ in the free $\di\As^{\I}$-algebra).
\end{proof}

One may easily see that every 
 $(\di\As^{\I})^!$-algebra
 is associative under the operation
\[
x*y := x\succ y + x\prec y,
\]
as happens for the smaller class 
$\pre\As = (\di\As)^!$ of {\em dendriform}
algebras.

In a very similar way, we may check the following statement.

\begin{proposition}[\cite{DMS2026}]\label{prop:diLie-I!}
The Koszul dual operad
$(\di\Lie^{\I})^!$ is generated by 
one operation 
satisfying the following relations:
\begin{equation}\label{Zinb1}
(x_1,x_2,x_3) - (x_3,x_2,x_1) = (x_2 x_3)x_1 -(x_2x_1)x_3,
\end{equation}
and
\begin{equation}\label{Zinb2}
x_1(x_2x_3) = x_2(x_1x_3).
\end{equation}
\end{proposition}
Here $(x,y,z)= (xy)z-x(yz)$ stands for the associator of elements $x,y,z$.

In particular, it follows from \eqref{Zinb1} and \eqref{Zinb2} that 
the operation $x*y = xy+yx$ in a $(\di\Lie^{\I})^!$-algebra is associative.

\section{The generalization of pre-Var algebras}

In this section, we
define initial pre-algebras as Koszul dual to initial dialgebras in the quadratic case. 
We also discuss operadic constructions of 
initial dialgebras and pre-algebras in order to expand the notion 
to arbitrary binary operads.

Recall that, given a multilinear variety 
$\Var $ of algebras with a family of bilinear operations $\mu_i$, $i\in I$, 
the variety of $\pre\Var$-algebras 
is constructed as follows \cite{GubKol_decorated2014}. 
An algebra $A$ with double family of operations 
$\mu_i^\prec$, $\mu_i^\succ$, $i\in I$,
is said to be a $\pre\Var$-algebra 
if for every $\Perm$-algebra $P$ the space $P\otimes A$ equipped with bilinear operations
\[
\mu_i(p\otimes a, q\otimes b)
 = pq\otimes \mu_i^\succ (a,b) + qp\otimes \mu_i^{\prec }(a,b)
\]
is a $\Var$-algebra. 

In \cite{BBGN13,GubKol-13}, an explicit algorithm for finding the identities of 
$\pre\Var$ was described. Suppose 
the identities of $\Var $ are generated by some subset $S$ of multilinear identities.
Then for every $f(x_1,\dots, x_n)\in S$ and for every $k=1,\dots, n$, consider the identity $f^{(k)}(x_1,\dots, x_k)$ obtained as follows. Suppose $f =\mu_i(g,h)$. 
If $x_k$ appears in $h$ then 
$f^{(k)} = \mu_i^\succ (g^{(*)}, h^{(k)})$, 
where $g^{(*)}$ is obtained from $g$ by replacing each $\mu_j$ with $\mu_j^* = \mu_j^\prec + \mu_j^\succ$. If $x_k$ appears in $g$ then, similarly, 
$f^{(k)} = \mu_i^\prec (g^{(k)}, h^{(*)})$.

The variety $\pre\Var$ is then defined by all 
$f^{(k)}(x_1,\dots, x_n)$, $k=1,\dots, n$, 
$f\in S$.

It follows from the definition (setting $P=\Bbbk $) that every $\pre\Var$-algebra 
with respect to the operations 
$\mu_i^*$, $i\in I$, is a $\Var$-algebra.

In the case when $\Var $ is a binary quadratic operad, we have 
$\pre\Var = \LSym\medbullet \Var$, 
the Manin black product of operads \cite{BBGN13}. Here $\LSym $ denotes the operad corresponding to the variety of left-symmetric algebras also called pre-Lie algebras. 

\begin{definition}\label{defn:pre-Var-initial}
Suppose $\Var $ is a binary quadratic operad. 
Then algebras from the variety corresponding to the operad $\pre\Var^{\I} := (\di(\Var^!)^{\I})^!$ 
are called {\em initial $\pre\Var$-algebras}. 
\end{definition}

Since the variety of initial $\di\Var$-algebras is embedded into the class of all $\di\Var$-algebras, the Koszul dual varieties are related in an inverse way:
every $\pre\Var$-algebra is an initial 
$\pre\Var$-algebra. 

\begin{example}\label{exmp:AsCom_initial}
Proposition~\ref{prop:diAS-I!} implies that
$\pre\As^{\I}$ is defined by \eqref{dend1} and \eqref{dend2}. Similarly, 
by Proposition~\ref{prop:diLie-I!}
$\pre\Com^{\I}$
is defined by \eqref{Zinb1} and \eqref{Zinb2}.
\end{example}

\begin{example}
Consider the variety $\Var=\NAP$
with one nonassociative right commutative 
operation: 
\[
(x_1x_2)x_3 = (x_1x_3)x_2.
\]
The Koszul dual operad $\NAP^!$ is 
defined by the identities
\[
(x_1x_2)x_3=(x_1x_3)x_2,\quad 
  x_1(x_2x_3)=0.
\]
\end{example}

The variety of $\di\NAP^{\I}$-algebras 
is defined by the zero-identities
\eqref{eq:Zero-dash}
along with 
\[
(x_1\dashv x_2)\dashv x_3 = (x_1\dashv x_3)\dashv x_2,
\quad 
(x_1\vdash x_2) \dashv x_3 = (x_1 \vdash x_3)\vdash x_2,
\]
and 
\[
(x_1\vdash x_2)\vdash x_3 = (x_1\vdash x_3)\vdash x_2.
\]

By Definition~\ref{defn:pre-Var-initial}, the variety of initial $\pre\NAP^{\I}$-algebras is Koszul dual to $\di\NAP^{\I}$. 
The latter is defined by
\[
(u*v)*w=(u*w)*v,\qquad u*(v*w)=0,
\]
where $x*y:=x\succ y+x\prec y$,
and
\[
(u \prec v) \prec w - (u \prec w) \prec v=0,\qquad
u \succ (w \prec v)=0,\qquad
u \succ (w \succ v)=0.
\]
Reducing the first two identities modulo the three additional relations above, we obtain
\[
(u\succ v)\succ w+(u\succ v)\prec w+(u\prec v)\succ w
-(u\succ w)\succ v-(u\succ w)\prec v-(u\prec w)\succ v=0
\]
and
\[
u\prec(v\succ w)+u\prec(v\prec w)=0.
\]

The aim of this section is to generalize Definition~\ref{defn:pre-Var-initial}
to an arbitrary binary operad $\Var$ to get 
a class of $\pre\Var^{\I}$-algebras which is more general than $\pre\Var$.

In the case $\Var=\Com$, the operad $\di\Com^{\I}$ coincides with the two-sided perm operad $\T\Perm$.
It is defined by associativity
together with the symmetry of the ternary product:
\[
x_1x_2x_3 = x_2x_1x_3 = x_1x_3x_2.
\]
Although in general (see \cite{DMS2026})
\[
\Var\medcirc\T\Perm \neq \di\Var^{\I},
\]
the operad $\T\Perm$ plays an essential role in defining the operad $\pre\Var^{\I}$.

Note that $\T\Perm^! = \Lie\Adm$, the operad of Lie-admissible algebras defined by the single identity 
\[
\sum\limits_{\sigma\in S_3}(-1)^\sigma 
(x_{\sigma(1)}, x_{\sigma(2)}, x_{\sigma(3)}) = 0.
\]

Suppose $\Var $ is a multi-linear variety of algebras with a family of binary operations $\mu_i$, $i\in I$. 
Without loss of generality, assume the family of operations is symmetric in the following sense: for every $i\in I$ there exists $j\in I$ such that $\mu_i^{(12)}=\pm \mu_j$. 
For example, the variety of associative algebras presented in this form has two operations $\mu (x,y)= xy$ and $\mu^{(12)}(x,y) = yx$.

In other words, one may consider the binary operad $\Var $ corresponding to this variety and choose $\mu_i$, $i\in I$, to be a linear basis of $\Var(2)$.

Consider the double family of operations $\mu_i^1,\mu_i^2$, $i\in I$, assume 
$(\mu_i^1)^{(12)} = \pm \mu_j^2$ provided that $\mu_i^{(12)}=\pm \mu_j$. In other words, we consider the tensor product of $S_2$-spaces
$\Var^*(2) = \Bbbk S_2\otimes \Var(2)$
so that $e\otimes \mu_i = \mu_i^2$, 
$(12)\otimes \mu_i = \mu_i^1$.

Denote by $\Var^*$ the variety of all algebras with operations from $\Var^*(2)$
such that the sums $\mu_i^* = \mu_i^1 +\mu_i^2$ satisfy the identities of $\Var $.

\begin{remark}
If $\Var $ is a binary quadratic operad then $\Var^* = \Lie\Adm\medbullet \Var $.
\end{remark}

It follows immediately from the algorithm of finding relations of the Manin black product 
of operads \cite{DongProp}.

By definition, the variety of $\di\Var^{\I}$-algebras is the intersection of classes
$\di\Var$ and $\Var^*$.
Hence, the T-ideal of $\di\Var^{\I}$
is a sum of the T-ideals for $\di\Var$ and 
for $\Var^*$.
Therefore, in order to get the identities of a dual variety, we need to consider the {\em intersection}
of T-ideals defining the dual classes.

The dual class to di-algebras is the variety 
$\pre\Var$ with operations $\mu_i^\succ$,
$\mu_i^\prec $, $i\in I$.
For the second class $\Var^*$, the dual construction is naturally based 
on the Manin white product of $\Var $ and $\Lie\Adm^! = \T\Perm$. 

\begin{lemma}
Let $\Var $ be a binary operad generated by 
$\mu_i$, $i\in I$, a basis of $\Var(2)$. Then $\Var\medcirc \T\Perm=\Var\otimes \T\Perm $
is generated by $\mu_i^\succ = \mu_i\otimes e, \mu_i^\prec = \mu_i\otimes (12)$, 
and $f$ is a relation of $\Var\medcirc \T\Perm$ if and only if its projection 
$\bar f$ on the first tensor factor 
is a relation of $\Var $.
\end{lemma}

\begin{proof}
For $n=2$, it follows from the symmetric structure on $\Var(2)\otimes \Bbbk S_2$.
For $n>2$ we have $\dim(\T\Perm(n))=1$, so 
$f(x_1\otimes p_1,\dots , x_n\otimes p_n)= 
\bar f(x_1,\dots, x_n)\otimes p_1\dots p_n$
for all $x_k$ in a $\Var$-algebra and for all $p_k$ in a $\T\Perm$-algebra. 
\end{proof}

Hence, $\Var\medcirc \T\Perm $ is defined by all those identities in $\mu_i^\succ$,
$\mu_i^\prec $ that remain valid 
in $\Var $ after removing all superscripts 
$\succ, \prec$.

This observation leads us to the following 

\begin{definition}\label{defn:preVar-InitialGeneral}
An algebra with operations 
$\mu_i^\succ, \mu_i^\prec $, $i\in I$, 
is an {\em initial $\pre\Var$-algebra} if 
it satisfies all multi-linear identities $f(x_1,\dots, x_n)$
of $\pre\Var$-algebras such that 
$\bar f$ holds on $\Var $.
\end{definition}

Denote by $\pre\Var^{\I}$ the class of all initial $\pre\Var$-algebras. 
In the case when $\Var$ is quadratic, this definition coincides 
with Definition~\ref{defn:pre-Var-initial}.

All the constructions discussed above can be summarized in the following diagram of T-ideals:

\begin{center}
\setlength{\unitlength}{1pt}
\begin{picture}(420,95)

\put(45,80){$T(\di\Var)$}
\put(45,35){$T(\pre\Var)$}
\put(150,80){$T(\di\Var^{\I})=T(\Var\medcirc\Perm)+T(\Var\medbullet\Lie\Adm)$}
\put(150,35){$T(\pre\Var^{\I})= T(\Var\medbullet\LSym)\cap T(\Var\medcirc\T\Perm)$}

\put(120,80){$\subset$}
\put(120,35){$\supset$}

\put(75,74){\vector(0,-1){24}}
\put(75,50){\vector(0, 1){24}}

\put(175,74){\vector(0,-1){24}}
\put(175,50){\vector(0, 1){24}}

\end{picture}
\end{center}

\vspace*{-\baselineskip}
\vspace*{-\baselineskip}

\begin{proposition}\label{prop:Star-prod}
In every $\pre\Var^{\I}$-algebra, the operations $\mu_i^*=\mu_i^\succ + \mu_i^{\prec}$, $i\in I$, satisfy the identities of $\Var $.
\end{proposition}

\begin{proof}
Let $f$ be a multilinear identity of $\Var $. 
It is well known that $f^*$ obtained from $f$ by replacing each $\mu_i$ with $\mu_i^*$  
is an identity of $\pre\Var$. On the other hand, 
$\overline{f^*} = 2^{n-1} f$, where $n=\deg f$.
Hence, $f^*$ holds on $\pre\Var^{\I}$. 
\end{proof}

Let us present the general algorithm of finding all defining relations of $\pre\Var^{\I}$ for a binary quadratic operad $\Var $.

Suppose $\mu_i$, $i\in I$, is a basis of $\Var(2)$, i.e., 
the family of operations (together with their opposites, as above).
Then the $S_3$-space $R$ of defining relations of $\Var $
lies in the $S_3$-module $E$ generated by all compositions
$\mu_i(\mu_j \otimes 1) = (x_1\circ_j x_2)\circ_ i x_3$.
Practically, 
\begin{equation}\label{eq:E-decomp}
E = E_1\oplus E_2\oplus E_3, 
\end{equation}
where $E_3$ is spanned by 
$\mu_i(\mu_j \otimes 1)$,
$E_2$
 is spanned by 
$\mu_i(\mu_j \otimes 1)^{(23)} = (x_1\circ_j x_3)\circ_i x_2$,
and
$E_1$
 is spanned by 
$\mu_i(\mu_j \otimes 1)^{(13)} = (x_3\circ_j x_2)\circ_i x_1$.
(It is enough to state only left-justified brackets since 
the operations $\mu_i^{(12)}$ also belong to the basis.)
In other words, $E_k$ is the linear span of all monomials of degree 3 with ``external'' letter $x_k$, $k=1,2,3$. In particular, 
\[
\pi_1(f) + \pi_2(f) + \pi_3(f) = f
\]
for every $f\in E$.

Consider the projections $\pi_k : R\to E_k$ along the decomposition
\eqref{eq:E-decomp}. Then for every $f\in R$ the set
$\pi_k^{-1}[\pi_k (f)]$ consists of all $h\in R$ with the same summand of monomials with external letter~$x_k$.

\begin{theorem}\label{identitiesPreI}
Suppose $\Var $ is a binary quadratic operad with the $S_3$-module of defining relations $R$. Then the space of defining relations for 
$\pre\Var^{\I}$ consists of 
\[
f_1^{(1)} + f_2^{(2)} + f_3^{(3)}, \quad f_k\in \pi_k^{-1}[\pi_k (f)] = f+\ker\pi_k,
\ f\in R,\ i=1,2,3.
\]
\end{theorem}
\begin{proof}
By the definition of dendriform splitting, 
$\overline{h^{(k)}} = h + \pi_k(h)$, for $k=1,2,3$.
Hence, if $\Phi = f_1^{(1)} + f_2^{(2)} + f_3^{(3)}$ 
as in the statement then 
\[
\overline{\Phi } = f_1 + \pi_1(f) + f_2 + \pi_2(f) + f_3 + \pi_3(f)
 = f_1+f_2+f_3+f \in R,
\]
so $\Phi $ holds in $\pre\Var^{\I}$. 

On the other hand, 
every $\Phi$ that holds on $\pre\Var^{\I}$ is of the form 
\[
\Phi = f_1^{(1)} + f_2^{(2)} + f_3^{(3)}
\]
for some $f_k\in R$. At the same time, $\overline{\Phi }$ 
holds of $\Var$, 
\[
\overline{\Phi } = f_1 + \pi_1(f_1) + f_2 + \pi_2(f_2) + f_3 + \pi_3(f_3) \in R,
\]
so 
$f = \pi_1(f_1) + \pi_2(f_2) + \pi_3(f_3) \in R$.
Therefore, 
$f_k\in \pi_k^{-1}[\pi_k(f)]$, as desired.
\end{proof}

In particular, if $f\in R\cap (E_2\oplus E_3) = \ker\pi_1$, 
i.e., $x_1$ is an ``inner'' letter of $f$, 
then one may choose $f_1=0$, $f_2=f_3=f$, 
so $f^{(2)}+f^{(3)}$ is an identity of $\pre\Var^{\I}$. Since $f^*=f^{(1)}+f^{(2)}+f^{(3)}$ also holds on $\pre\Var^{\I} $
by Proposition~\ref{prop:Star-prod}, 
we have $f^{(1)}$ is an identity on $\pre\Var^{\I}$. 
Similarly, 
if $f\in R\cap (E_i\oplus E_j)$ then 
$f^{(k)}$ holds on $\pre\Var^{\I}$ for 
the remaining $k\in \{1,2,3\}\setminus \{i,j\}$.

As a result, we approach an algorithm to deduce the defining relations of $\pre\Var^{\I}$ for a binary quadratic operad $\Var $. First, for each defining relation $f$ of $\Var$ consider $f^*$. 
Second, for each $f\in \ker\pi_k$, $k=1,2,3$, 
consider $f^{(k)}$.
Note that in the second step one has to find a linear basis of $\ker\pi_1$ modulo the transposition $(23)$
since $R$ is an $S_3$-module.

\begin{example}\label{exmp:Lie}
For $\Var=\Lie$, the space $R$ is spanned by
\[
x_1(x_2x_3)+x_2(x_3x_1)+x_3(x_1x_2)
\]
in the 3-dimensional space of monomials, 
so  $\ker\pi_1=\{0\}$.
Hence, $\pre\Lie^{\I}$ coincides with 
the operad of Lie-admissible algebras.
\end{example}

\begin{example}\label{exmp:Nilp}
Suppose $\Var=\mathrm{RNilp}$ is generated by non-commutative operation with one defining relation 
$x_1(x_2x_3)=0$. 
Then a basis of $\ker\pi_1$ modulo $(23)$ consists of two polynomials
$g=x_2(x_1x_3)$ and $h=x_2(x_3x_1)$.
Thus, 
$\pre\mathrm{RNilp}^{\I}$ 
is defined by 
\[
x_1*(x_2*x_3),\quad x_2\succ (x_1\prec x_3), 
x_2\succ (x_3\succ x_1).
\]
\end{example}

Here, as above, $x*y = x\succ y + x\prec y$.

\begin{example}\label{exmp:Alt}
Let 
$\Var=\Alt$ be the operad of alternative algebras generated by one non-commutative operation satisfying 
\[
(x_1,x_2,x_3) = -(x_2,x_1,x_3),\quad 
(x_1,x_2,x_3)=-(x_1,x_3,x_2).
\]
Here $(a,b,c)=(ab)c-a(bc)$ is the associator. Alternative-type identities and
their relation to binary perm algebras were recently considered in \cite{NAP}.
It is well known that 
$\dim R=5$, and it is straightforward to find that $\dim\ker\pi_1=1$. 
For example,
$f=(x_2,x_1,x_3)+(x_3,x_1,x_2)\in R$, 
so this is a basis of $\ker\pi_1$.
Hence, 
$\pre\Alt^{\I}$
is generated by two operations $x_1\succ x_2$, 
$x_1\prec x_2$ 
satisfying 
\[
(x_1,x_2,x_3)^*+(x_3,x_2,x_1)^*,
\quad 
(x_1,x_2,x_3)^{(2)}+(x_3,x_2,x_1)^{(2)},
\quad
(x_1,x_2,x_3)^*+(x_1,x_3,x_2)^*.
\]
\end{example}

\begin{example}\label{exmp:Nov}
For $\Var=\Nov$, the operad of Novikov algebras,
the $S_3$-module $R$ is generated by  
\[
(x_1,x_2,x_3) = (x_2,x_1,x_3),
\quad 
(x_1 x_2)x_3=(x_1 x_3)x_2.
\]
It is well known that $\dim R=6$, and it is easy to find that $\dim\ker\pi_1=2$,
\[
g = (x_2x_1)x_3 - x_2(x_3x_1)+(x_3x_1)x_2-x_3(x_2x_1),
\quad 
h=(x_1x_2)x_3 - (x_1x_3)x_2
\]
form a linear basis of $\ker\pi_1$.
Hence, $\pre\Nov^{\I}$ is defined by 
\[
\begin{gathered}
(x_1,x_2,x_3)^* = (x_2,x_1,x_3)^*,
\quad 
(x_1 *x_2)*x_3=(x_1 *x_3)*x_2, \\
g^{(1)} = (x_2\succ x_1)\prec x_3 - x_2\succ (x_3\succ x_1)+(x_3\succ x_1)\prec x_2-x_3\succ (x_2\succ x_1),
\\
h^{(1)}=(x_1\prec x_2)\prec x_3 - (x_1\prec x_3)\prec x_2.
\end{gathered}
\]

Indeed, the operad $\di\Nov^\I$ is defined by 
\[
\begin{gathered}
(x_1\dashv x_2)\vdash x_3=(x_1\vdash x_2)\vdash x_3,\;
x_1\dashv(x_2\vdash x_3)=x_1\dashv(x_2\dashv x_3),\\
(x_1\vdash x_2)\vdash x_3-x_1\vdash(x_2\vdash x_3)
=
(x_2\vdash x_1)\vdash x_3-x_2\vdash(x_1\vdash x_3),\\
(x_1\dashv x_2)\dashv x_3-x_1\dashv(x_2\dashv x_3)=
(x_2\vdash x_1)\dashv x_3-x_2\vdash(x_1\dashv x_3),\\
(x_1\dashv x_2)\dashv x_3=(x_1\dashv x_3)\dashv x_2,\;
(x_1\vdash x_2)\dashv x_3=(x_1\vdash x_3)\vdash x_2,\\
(x_1\vdash x_2)\vdash x_3-x_1\dashv(x_2\dashv x_3)=
(x_2\vdash x_1)\vdash x_3-x_2\dashv(x_1\dashv x_3),\\
(x_1\vdash x_2)\vdash x_3=(x_1\vdash x_3)\vdash x_2.
\end{gathered}
\]
For Theorem \ref{identitiesPreI}, it follows that
\[
\langle R_{\di\Nov^{\I}},R_{\pre\Nov^{\I}}\rangle=0.
\]
It remains to note that
\[
\dim(\di\Nov^{\I}(3))+\dim(\pre\Nov^{\I}(3))=12+36=48,
\]
which implies $(\di\Nov^{\I})^!=\pre\Nov^{\I}$.
\end{example}

In some cases, the operad $\pre\Var^{\I}$ may be presented as just a Manin product of $\Var$ and another fixed operad.

\begin{definition}
Let $\Var $ be a binary operad. Let us say $\Var $ is {\em good} if $\Var\medcirc \Zinb = \pre\Var$. (In general, we have just a morphism 
$\pre\Var\to \Var\medcirc \Zinb$.)
A binary operad $\Var$ is called {\em doog}, if 
$\Var\medbullet \Leib = \di\Var$. 
\end{definition}

In the quadratic case, $\Var$ is good if and only if $\Var^!$ is doog. For example, the operads $\As$, $\Lie$, $\Com$, $\Alt$, $\Nov $
are good and doog, the operad $\NAP$ is not good, but $\NAP^!$ is good.

\begin{proposition}\label{prop:doog-diVar}
If $\Var$ is 
a binary quadratic operad which is doog
then $\di\Var^{\I} = \di\Lie^{\I}\medbullet \Var$.
\end{proposition}

\begin{proof}
Recall that $\di\Lie=\Leib$.
Moreover, by definition of initial Lie dialgebras, the operad
$\di\Lie^{\I}$ is obtained by imposing, in addition to the Leibniz identity,
the Lie-admissibility identity. Hence
\[
\di\Lie^{\I}=\Leib\cap\Lie\Adm .
\]
Equivalently, at the level of defining $T$-ideals,
\[
T(\di\Lie^{\I})=T(\Leib)+T(\Lie\Adm).
\]

Since the Manin black product is
linear on the quadratic relation space, we obtain
\[
\di\Lie^{\I}\medbullet\Var
=
(\Leib\cap\Lie\Adm)\medbullet\Var
=
(\Leib\medbullet\Var)\cap(\Lie\Adm\medbullet\Var).
\]

Since $\Var$ is doog, and by the definition of the operad $\Var^*$ for a binary
quadratic operad,
\[
\Var^*=\Lie\Adm\medbullet\Var.
\]
we have
\[
\di\Lie^{\I}\medbullet\Var=\di\Var\cap\Var^*.
\]

Finally, the class of initial $\di\Var$-algebras is precisely the intersection
of the class of $\di\Var$-algebras with the class $\Var^*$.
\end{proof}


\begin{corollary}
If a binary quadratic operad $\Var $
is good then 
\[
\pre\Var^{\I} = (\di(\Var^!)^{\I})^! = 
(\di\Lie^{\I}\medbullet \Var^!)^! = \pre\Com^{\I} \medcirc \Var .
\]
\end{corollary}

\section{Free initial pre-associative and pre-commutative algebras}

In this section, we describe the free initial pre-associative and pre-commutative
algebras
\[
\pre\text{-}\As^{\I}\langle X\rangle
\quad \text{and} \quad
\pre\text{-}\Com^{\I}\langle X\rangle,
\]
respectively. Since the operad $\pre\text{-}\As^{\I}$ is non-symmetric, we regard
it as an operad defined by a system of rewriting rules. For more details on non-symmetric operads, see \cite{Loday,MashSart2026}. More precisely,
$\pre\text{-}\As^{\I}$ is generated by two binary operations, denoted by $x$ and
$y$, corresponding to $\prec$ and $\succ$, respectively, and is subject to the
following rewriting rules:
\begin{equation}\label{InPreAs1}
x\bigl(y(*\; *)\; *\bigr)
\rightsquigarrow
y\bigl(*\; x(*\; *)\bigr),
\end{equation}
and
\begin{equation}\label{InPreAs2}
\begin{aligned}
y\bigl(y(*\; *)\; *\bigr)
\rightsquigarrow{}&
-\,y\bigl(x(*\; *)\; *\bigr)
-x\bigl(x(*\; *)\; *\bigr)  \\
&+y\bigl(*\; y(*\; *)\bigr)
+x\bigl(*\; y(*\; *)\bigr)
+x\bigl(*\; x(*\; *)\bigr).
\end{aligned}
\end{equation}

\begin{lemma}
In the operad $\pre\text{-}\As^{\I}$, the following identity holds:
\begin{multline*}
x\bigl(x(x(*\; *)\; *)\; *\bigr)=
x\bigl(x(*\; y(*\; *))\; *\bigr)
+x\bigl(x(*\; x(*\; *))\; *\bigr)
+x\bigl(x(*\; *)\; x(*\; *)\bigr)\\
-\,x\bigl(*\; y(*\; x(*\; *))\bigr)
-x\bigl(*\; x(*\; x(*\; *))\bigr).
\end{multline*}
\end{lemma}

\begin{proof}
Consider the critical self-composition of \eqref{InPreAs2} corresponding to the
monomial
\[
y\bigl(y(y(*\; *)\; *)\; *\bigr).
\]
Reducing this monomial in the two possible ways by using \eqref{InPreAs2}, and
then applying \eqref{InPreAs1} and \eqref{InPreAs2} to the resulting terms, we
obtain precisely the stated identity.
\end{proof}

\begin{theorem}
The rewriting rules \eqref{InPreAs1} and \eqref{InPreAs2}, together with the rule
\begin{multline}\label{InPreAs3}
x\bigl(x(x(*\; *)\; *)\; *\bigr)
\rightsquigarrow
x\bigl(x(*\; y(*\; *))\; *\bigr)
+x\bigl(x(*\; x(*\; *))\; *\bigr)
+x\bigl(x(*\; *)\; x(*\; *)\bigr)\\
-\,x\bigl(*\; y(*\; x(*\; *))\bigr)
-x\bigl(*\; x(*\; x(*\; *))\bigr),
\end{multline}
form a Gr\"obner--Shirshov basis in the magma operad.
\end{theorem}

\begin{proof}
It is enough to check all critical compositions of the leading monomials of
\eqref{InPreAs1}--\eqref{InPreAs3}.

In degree $4$, there are two critical compositions:
\begin{itemize}
    \item the composition of \eqref{InPreAs1} with \eqref{InPreAs2} at the
    monomial
    \[
    x\bigl(y(y(*\; *)\; *)\; *\bigr);
    \]
    
    \item the self-composition of \eqref{InPreAs2} at the monomial
    \[
    y\bigl(y(y(*\; *)\; *)\; *\bigr).
    \]
\end{itemize}
In both cases, after applying \eqref{InPreAs1} and \eqref{InPreAs2}, the
corresponding composition reduces to the relation \eqref{InPreAs3}. Hence all
degree $4$ compositions are trivial modulo
\eqref{InPreAs1}--\eqref{InPreAs3}.

In degree $5$, there are two remaining critical compositions:
\begin{itemize}
    \item the composition of \eqref{InPreAs1} with \eqref{InPreAs3} at the
    monomial
    \[
    x\bigl(x(x(y(*\; *)\; *)\; *)\; *\bigr);
    \]
    
    \item the self-composition of \eqref{InPreAs3} at the monomial
    \[
    x\bigl(x(x(x(*\; *)\; *)\; *)\; *\bigr).
    \]
\end{itemize}
A direct reduction by \eqref{InPreAs1}--\eqref{InPreAs3} shows that both
compositions are trivial.

Finally, in degree $6$, the only remaining composition is the self-composition
of \eqref{InPreAs3} at the monomial
\[
x\bigl(x(x(x(x(*\; *)\; *)\; *)\; *)\; *\bigr).
\]
Reducing this monomial by \eqref{InPreAs1}--\eqref{InPreAs3}, we again obtain
zero. Therefore all critical compositions are trivial, and the rewriting system
\eqref{InPreAs1}--\eqref{InPreAs3} is a Gr\"obner--Shirshov basis.
\end{proof}

To describe a basis of the free algebra
$\pre\text{-}\As^{\I}\langle x_1\rangle$, we consider planar binary trees whose
leaves are unlabeled and whose internal vertices are labeled by $\bullet$ or
$\circ$. We call such a tree \emph{admissible} if it contains no subtree of one
of the following forms:

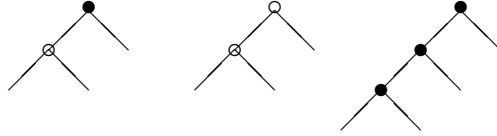
\begin{picture}(500,80)
\put(147,68){$\bullet$}
\put(150,70){\line(-1,-1){15}}
\put(150,70){\line(1,-1){15}}
\put(132,52){$\circ$}
\put(135,55){\line(-1,-1){15}}
\put(135,55){\line(1,-1){15}}

\put(217,68){$\circ$}
\put(220,70){\line(-1,-1){15}}
\put(220,70){\line(1,-1){15}}
\put(202,52){$\circ$}
\put(205,55){\line(-1,-1){15}}
\put(205,55){\line(1,-1){15}}

\put(287,68){$\bullet$}
\put(290,70){\line(-1,-1){15}}
\put(290,70){\line(1,-1){15}}
\put(272,52){$\bullet$}
\put(275,55){\line(-1,-1){15}}
\put(275,55){\line(1,-1){15}}
\put(257,37){$\bullet$}
\put(260,40){\line(-1,-1){15}}
\put(260,40){\line(1,-1){15}}
\end{picture}

\vspace*{-\baselineskip}
\vspace*{-\baselineskip}

Let $\mathcal{A}$ denote the set of monomials corresponding to admissible
trees.

\begin{corollary}
The set $\mathcal{A}$ is a linear basis of the algebra
$\pre\text{-}\As^{\I}\langle x_1\rangle$.
\end{corollary}

\begin{example}
In degree $3$, that is, for planar binary trees with $3$ leaves, the admissible
monomials are
\[
\mathcal{A}_3=\Bigl\{
\bullet(\bullet(*\;*),\;*),\ 
\circ(\bullet(*\;*),\;*),\ 
\bullet(*,\;\bullet(*\;*)),\ 
\bullet(*,\;\circ(*\;*)),\ 
\circ(*,\;\bullet(*\;*)),\ 
\circ(*,\;\circ(*\;*))
\Bigr\}.
\]

In degree $4$, that is, for planar binary trees with $4$ leaves, the admissible
monomials are
\begingroup
\small
\setlength{\arraycolsep}{4pt}
\renewcommand{\arraystretch}{1.35}
\[
\mathcal{A}_4=\left\{
\begin{array}{lll}
\bullet(\bullet(*,\;\bullet(*\;*)),\;*),
&
\bullet(\bullet(*,\;\circ(*\;*)),\;*),
&
\circ(\bullet(\bullet(*\;*),\;*),\;*) ,
\\
\circ(\bullet(*,\;\bullet(*\;*)),\;*),
&
\circ(\bullet(*,\;\circ(*\;*)),\;*),
&
\bullet(\bullet(*\;*),\;\bullet(*\;*)) ,
\\
\bullet(\bullet(*\;*),\;\circ(*\;*)),
&
\circ(\bullet(*\;*),\;\bullet(*\;*)),
&
\circ(\bullet(*\;*),\;\circ(*\;*)) ,
\\
\bullet(*,\;\bullet(\bullet(*\;*),\;*)),
&
\bullet(*,\;\circ(\bullet(*\;*),\;*)),
&
\circ(*,\;\bullet(\bullet(*\;*),\;*)) ,
\\
\bullet(*,\;\bullet(*,\;\bullet(*\;*))),
&
\bullet(*,\;\bullet(*,\;\circ(*\;*))),
&
\bullet(*,\;\circ(*,\;\bullet(*\;*))) ,
\\
\bullet(*,\;\circ(*,\;\circ(*\;*))),
&
\circ(*,\;\bullet(*,\;\bullet(*\;*))),
&
\circ(*,\;\bullet(*,\;\circ(*\;*))) ,
\\
\circ(*,\;\circ(*,\;\bullet(*\;*))),
&
\circ(*,\;\circ(*,\;\circ(*\;*))),
&
\circ(*,\;\circ(\bullet(*\;*),\;*))
\end{array}
\right\}.
\]
\endgroup
\noindent
In degree $5$ (i.e., for planar binary trees with $5$ leaves), the admissible trees form the set

\begingroup
\scriptsize
\setlength{\arraycolsep}{2pt}
\renewcommand{\arraystretch}{1.25}
\[
\mathcal{A}_5=\left\{
\begin{array}{@{}llll@{}}
\bullet(*,\bullet(*,\bullet(*,\bullet(*\,*)))) ,
&
\bullet(*,\bullet(*,\bullet(*,\circ(*\,*)))) ,
&
\bullet(*,\bullet(*,\bullet(\bullet(*\,*),*))) ,
&
\bullet(*,\bullet(*,\circ(*,\bullet(*\,*)))) ,
\\
\bullet(*,\bullet(*,\circ(*,\circ(*\,*)))) ,
&
\bullet(*,\bullet(*,\circ(\bullet(*\,*),*))) ,
&
\bullet(*,\bullet(\bullet(*\,*),\bullet(*\,*))) ,
&
\bullet(*,\bullet(\bullet(*\,*),\circ(*\,*))) ,
\\
\bullet(*,\bullet(\bullet(*,\bullet(*\,*)),*)) ,
&
\bullet(*,\bullet(\bullet(*,\circ(*\,*)),*)) ,
&
\bullet(*,\circ(*,\bullet(*,\bullet(*\,*)))) ,
&
\bullet(*,\circ(*,\bullet(*,\circ(*\,*)))) ,
\\
\bullet(*,\circ(*,\bullet(\bullet(*\,*),*))) ,
&
\bullet(*,\circ(*,\circ(*,\bullet(*\,*)))) ,
&
\bullet(*,\circ(*,\circ(*,\circ(*\,*)))) ,
&
\bullet(*,\circ(*,\circ(\bullet(*\,*),*))) ,
\\
\bullet(*,\circ(\bullet(*\,*),\bullet(*\,*))) ,
&
\bullet(*,\circ(\bullet(*\,*),\circ(*\,*))) ,
&
\bullet(*,\circ(\bullet(*,\bullet(*\,*)),*)) ,
&
\bullet(*,\circ(\bullet(*,\circ(*\,*)),*)) ,
\\
\bullet(*,\circ(\bullet(\bullet(*\,*),*),*)) ,
&
\bullet(\bullet(*\,*),\bullet(*,\bullet(*\,*))) ,
&
\bullet(\bullet(*\,*),\bullet(*,\circ(*\,*))) ,
&
\bullet(\bullet(*\,*),\bullet(\bullet(*\,*),*)) ,
\\
\bullet(\bullet(*\,*),\circ(*,\bullet(*\,*))) ,
&
\bullet(\bullet(*\,*),\circ(*,\circ(*\,*))) ,
&
\bullet(\bullet(*\,*),\circ(\bullet(*\,*),*)) ,
&
\bullet(\bullet(*,\bullet(*\,*)),\bullet(*\,*)) ,
\\
\bullet(\bullet(*,\bullet(*\,*)),\circ(*\,*)) ,
&
\bullet(\bullet(*,\bullet(*,\bullet(*\,*))),*) ,
&
\bullet(\bullet(*,\bullet(*,\circ(*\,*))),*) ,
&
\bullet(\bullet(*,\bullet(\bullet(*\,*),*)),*) ,
\\
\bullet(\bullet(*,\circ(*\,*)),\bullet(*\,*)) ,
&
\bullet(\bullet(*,\circ(*\,*)),\circ(*\,*)) ,
&
\bullet(\bullet(*,\circ(*,\bullet(*\,*))),*) ,
&
\bullet(\bullet(*,\circ(*,\circ(*\,*))),*) ,
\\
\bullet(\bullet(*,\circ(\bullet(*\,*),*)),*) ,
&
\circ(*,\bullet(*,\bullet(*,\bullet(*\,*)))) ,
&
\circ(*,\bullet(*,\bullet(*,\circ(*\,*)))) ,
&
\circ(*,\bullet(*,\bullet(\bullet(*\,*),*))) ,
\\
\circ(*,\bullet(*,\circ(*,\bullet(*\,*)))) ,
&
\circ(*,\bullet(*,\circ(*,\circ(*\,*)))) ,
&
\circ(*,\bullet(*,\circ(\bullet(*\,*),*))) ,
&
\circ(*,\bullet(\bullet(*\,*),\bullet(*\,*))) ,
\\
\circ(*,\bullet(\bullet(*\,*),\circ(*\,*))) ,
&
\circ(*,\bullet(\bullet(*,\bullet(*\,*)),*)) ,
&
\circ(*,\bullet(\bullet(*,\circ(*\,*)),*)) ,
&
\circ(*,\circ(*,\bullet(*,\bullet(*\,*)))) ,
\\
\circ(*,\circ(*,\bullet(*,\circ(*\,*)))) ,
&
\circ(*,\circ(*,\bullet(\bullet(*\,*),*))) ,
&
\circ(*,\circ(*,\circ(*,\bullet(*\,*)))) ,
&
\circ(*,\circ(*,\circ(*,\circ(*\,*)))) ,
\\
\circ(*,\circ(*,\circ(\bullet(*\,*),*))) ,
&
\circ(*,\circ(\bullet(*\,*),\bullet(*\,*))) ,
&
\circ(*,\circ(\bullet(*\,*),\circ(*\,*))) ,
&
\circ(*,\circ(\bullet(*,\bullet(*\,*)),*)) ,
\\
\circ(*,\circ(\bullet(*,\circ(*\,*)),*)) ,
&
\circ(*,\circ(\bullet(\bullet(*\,*),*),*)) ,
&
\circ(\bullet(*\,*),\bullet(*,\bullet(*\,*))) ,
&
\circ(\bullet(*\,*),\bullet(*,\circ(*\,*))) ,
\\
\circ(\bullet(*\,*),\bullet(\bullet(*\,*),*)) ,
&
\circ(\bullet(*\,*),\circ(*,\bullet(*\,*))) ,
&
\circ(\bullet(*\,*),\circ(*,\circ(*\,*))) ,
&
\circ(\bullet(*\,*),\circ(\bullet(*\,*),*)) ,
\\
\circ(\bullet(*,\bullet(*\,*)),\bullet(*\,*)) ,
&
\circ(\bullet(*,\bullet(*\,*)),\circ(*\,*)) ,
&
\circ(\bullet(*,\bullet(*,\bullet(*\,*))),*) ,
&
\circ(\bullet(*,\bullet(*,\circ(*\,*))),*) ,
\\
\circ(\bullet(*,\bullet(\bullet(*\,*),*)),*) ,
&
\circ(\bullet(*,\circ(*\,*)),\bullet(*\,*)) ,
&
\circ(\bullet(*,\circ(*\,*)),\circ(*\,*)) ,
&
\circ(\bullet(*,\circ(*,\bullet(*\,*))),*) ,
\\
\circ(\bullet(*,\circ(*,\circ(*\,*))),*) ,
&
\circ(\bullet(*,\circ(\bullet(*\,*),*)),*) ,
&
\circ(\bullet(\bullet(*\,*),*),\bullet(*\,*)) ,
&
\circ(\bullet(\bullet(*\,*),*),\circ(*\,*)) ,
\\
\circ(\bullet(\bullet(*\,*),\bullet(*\,*)),*) ,
&
\circ(\bullet(\bullet(*\,*),\circ(*\,*)),*) ,
&
\circ(\bullet(\bullet(*,\bullet(*\,*)),*),*) ,
&
\circ(\bullet(\bullet(*,\circ(*\,*)),*),*)
\end{array}
\right\}.
\]
\endgroup
\end{example}

Proceeding in the same way, we obtain
\[
|\mathcal{A}_6|=322,\qquad
|\mathcal{A}_7|=1347,\qquad
|\mathcal{A}_8|=5798,\qquad
|\mathcal{A}_9|=25512,\qquad
|\mathcal{A}_{10}|=114236.
\]

Consequently, the dimensions of the non-symmetric operad
$\pre\text{-}\As^{\I}$ are given by
\[
\begin{array}{c|cccccccccc}
n & 1 & 2 & 3 & 4 & 5 & 6 & 7 & 8 & 9 & 10 \\ \hline
\dim(\pre\text{-}\As^{\I}(n))
& 1 & 2 & 6 & 21 & 80 & 322 & 1347 & 5798 & 25512 & 114236
\end{array}
\]

These values coincide with the entries of the OEIS sequence A106228 for
$n\leq 10$.
Recall that A106228 counts permutations of length $n$ avoiding the partially
ordered pattern
\[
\{1>3,\;1>4,\;4>2\}
\]
of length $4$. Equivalently, it counts permutations of length $n$ which contain
no subsequence of length $4$ such that the first entry is the largest one and
the fourth entry is greater than the second entry. This sequence admits the
following explicit formula:
\[
a(n)=\frac{1}{n+1}
\sum_{k=0}^{n}
\binom{n+1}{k}
\binom{n+k+1}{n-k}.
\]

We now turn to the free initial pre-commutative algebra
$\pre\text{-}\Com^{\I}\langle X\rangle$. At present, an explicit description of
a linear basis of $\pre\text{-}\Com^{\I}\langle X\rangle$ remains an open
problem. Nevertheless, using the computer algebra package \cite{Albert}, we
computed the dimensions in low degrees.

For the operad $\pre\text{-}\Com^{\I}$, we obtain
\[
\begin{array}{c|cccccccccc}
n & 1 & 2 & 3 & 4 & 5 & 6 & 7 & 8 & 9 & 10 \\ \hline
\dim(\pre\text{-}\Com^{\I}(n))
& 1 & 2 & 7 & 32 & 181 & 1232 & 9787 & 88832 & 907081 & 10291712
\end{array}
\]
These values coincide with the entries of the OEIS sequence A006154 for
$n\leq 10$.

Recall that A006154 counts labeled ordered partitions of an $n$-element set
into blocks of odd size. Its exponential generating function is
\[
f(x)=\frac{1}{1-\sinh(x)}.
\]
An explicit formula for this sequence is
\[
a(n)=
\sum_{k=1}^{n}
\sum_{i=0}^{k}
(-1)^i
\frac{(k-2i)^n}{2^k}
\binom{k}{i}.
\]

\begin{remark}
There is no need to construct a linear basis of the free initial pre-Lie algebra. As shown in \cite{DMS2026},
the algebra $\pre\text{-}\Lie^{\I}\langle X\rangle$ coincides with the free
Lie-admissible algebra.
\end{remark}

\section{A bijection of the set \texorpdfstring{$\mathcal{A}$}{A} with the paths of type \texorpdfstring{$F$}{F}}

In this section, we prove that
\[
\dim\bigl(\pre\text{-}\As^{\I}(n)\bigr)=A106228(n-1),\qquad n\ge 1.
\]
The sequence A106228 enumerates lattice paths of type $F$, as shown in \cite{HuhKimSeoShin2024}. We recall the definition of these paths below.

Throughout, we use the natural convention that the arity of a monomial is the
number of leaves of the corresponding planar binary tree.
We write $\mathcal A_n$ for the set of admissible planar binary trees with $n$
leaves and internal labels $\bullet,\circ$, that is, trees avoiding the forbidden
subtrees corresponding to the leading monomials of \eqref{InPreAs1}--\eqref{InPreAs3}.

\begin{definition}
Following \cite{HuhKimSeoShin2024}, let $F_{m}$ be the set of all lattice paths
\[
Q=((0,0),(x_1,y_1),\dots,(x_m,y_m))
\]
with steps in
\[
F=\{(a,b)\mid a\ge 1,\ b\le 1\}\cup\{(0,1)\},
\]
such that $x_i\le y_i$ for all $i$.
We call such paths \emph{$F$-paths}.
Their height is
\[
\operatorname{ht}(Q)=y_m-x_m.
\]
\end{definition}

To construct a bijection
\[
\mathcal A_n\longleftrightarrow F_{n-1},
\]
we first introduce several auxiliary definitions and lemmas.

Let $T\in\mathcal A_n$.
The \emph{rightmost leaf} of $T$ is the unique leaf obtained by following always
the right edge from the root.

\begin{definition}
For admissible trees $T$ and $S$, define their \emph{direct sum} $T\oplus S$ to
be the admissible tree obtained from $T$ by replacing the rightmost leaf of $T$
with the tree $\circ(*,S)$.
\end{definition}

\begin{lemma}\label{lem:sum-assoc}
Let $T$ and $S$ be admissible trees. Then the direct sum $T\oplus S$ is again admissible. Moreover, the operation $\oplus$ is associative.
\end{lemma}

\begin{proof}
Replacing the rightmost leaf of $T$ by $\circ(*,S)$ creates only one new internal
vertex, namely a vertex labeled $\circ$ whose left child is a leaf. Since the new
vertex lies on the rightmost branch, no forbidden subtree of the forms
\[
\bullet(\circ(\cdot,\cdot),\cdot),\qquad
\circ(\circ(\cdot,\cdot),\cdot),\qquad
\bullet(\bullet(\bullet(\cdot,\cdot),\cdot),\cdot)
\]
can be created. Hence $T\oplus S$ is admissible.

Associativity is immediate, because the rightmost leaf of $T\oplus S$ lies in the
copy of $S$. Therefore replacing it by $\circ(*,U)$ gives precisely
\[
(T\oplus S)\oplus U=T\oplus (S\oplus U).
\]
\end{proof}

\begin{definition}
An admissible tree is called \emph{indecomposable} if it cannot be represented in the form
$U\oplus V$.
For $T\in\mathcal A_n$, we denote by $\bl(T)$ the \emph{block number} of $T$, that is, the number of indecomposable factors in the unique decomposition of $T$ as an iterated direct sum.
\end{definition}

\begin{lemma}\label{lem:block-criterion}
A tree $T$ is indecomposable if and only if no vertex on the right spine of $T$
is labeled $\circ$ and has a left leaf.
Moreover, every admissible tree admits a unique decomposition
\[
T=T_1\oplus T_2\oplus\cdots\oplus T_r
\]
into indecomposable admissible trees, and then $\bl(T)=r$. 

For any admissible trees $U,R$, one has 
\[
\bl(U\oplus R)=\bl(U)+\bl(R).
\]

\end{lemma}

\begin{proof}
By definition of $\oplus$, each application of $\oplus$ inserts exactly one new
vertex of the form $\circ(*,\cdot)$ on the right spine. Conversely, every such
vertex on the right spine separates the tree uniquely into a left summand and a
right summand. Cutting successively at all such vertices yields the required
decomposition, and uniqueness is obvious.

To prove the second part, we write $U=U_1\oplus\cdots\oplus U_p$ and $R=R_1\oplus\cdots\oplus R_q$ as indecomposable blocks (unique by Lemma~\ref{lem:block-criterion}).  By the definition of $\oplus$ and associativity, 
\[
U\oplus R=U_1\oplus\cdots\oplus U_p\oplus R_1\oplus\cdots\oplus R_q,
\] 
so $\bl(U\oplus R)=p+q=\bl(U)+\bl(R)$.  Uniqueness of block decomposition ensures this is well-defined.

\end{proof}

\begin{lemma}\label{lem:five-types}
Every non-leaf admissible tree belongs to exactly one of the following five disjoint classes:
\begin{align*}
\textup{(I)}\;& \circ(*,R),\\
\textup{(II)}\;& \bullet(*,R),\\
\textup{(III)}\;& \bullet(\bullet(*,U),R),\\
\textup{(IV)}\;& \circ(\bullet(*,U),R),\\
\textup{(V)}\;& \circ(\bullet(\bullet(*,U_1),U_2),R),
\end{align*}
where $R,U,U_1,U_2$ are admissible trees.
\end{lemma}

\begin{proof}
Let $T=(\alpha,L,R)$ be a non-leaf admissible tree, where $\alpha\in\{\bullet,\circ\}$.

If $L$ is a leaf, then we are in \textup{(I)} or \textup{(II)} depending on
$\alpha=\circ$ or $\alpha=\bullet$.

Assume now that $L$ is not a leaf. Since neither
$\bullet(\circ(\cdot,\cdot),\cdot)$ nor $\circ(\circ(\cdot,\cdot),\cdot)$ is admissible,
the root of $L$ must be labeled $\bullet$. Thus $L=\bullet(L_1,L_2)$.

If $\alpha=\bullet$, then the forbidden subtree
$\bullet(\bullet(\bullet(\cdot,\cdot),\cdot),\cdot)$ shows that $L_1$ must be a leaf.
Hence we are in \textup{(III)}.

If $\alpha=\circ$ and $L_1$ is a leaf, then we are in \textup{(IV)}.

Finally, if $\alpha=\circ$ and $L_1$ is not a leaf, then admissibility of $L$
forces $L_1=\bullet(*,U_1)$ for some admissible $U_1$, and $L_2=U_2$ is arbitrary.
Hence we are in \textup{(V)}.

The five classes are pairwise disjoint by construction.
\end{proof}

\begin{lemma}\label{lem:block-rec}
For admissible trees, the block number satisfies:
\begin{align*}
\bl(*)&=1,\\
\bl(\circ(*,R))&=\bl(R)+1,\\
\bl(\bullet(*,R))&=\bl(R),\\
\bl(\bullet(\bullet(*,U),R))&=\bl(R),\\
\bl(\circ(\bullet(*,U),R))&=\bl(R),\\
\bl(\circ(\bullet(\bullet(*,U_1),U_2),R))&=\bl(R).
\end{align*}
\end{lemma}

\begin{proof}
Only in the case $\circ(*,R)$ does the root itself contribute a new separator
$\circ(*,\cdot)$ on the right spine. In all other four cases, the root is not such
a separator, so the block number is inherited from the right subtree.
\end{proof}

\begin{definition}
Define recursively a map
\[
\Phi:\bigsqcup_{n\ge 1}\mathcal A_n\longrightarrow \bigsqcup_{m\ge 0}F_m
\]
as follows. For the one-leaf tree,
\[
\Phi(*)=\varnothing.
\]
For a non-leaf admissible tree, we use the five cases of
Lemma~\ref{lem:five-types}:

\begin{align*}
\Phi(\circ(*,R))
&=\Phi(R)\,(0,1),\\
\Phi(\bullet(*,R))
&=\Phi(R)\,(1,1),\\
\Phi(\bullet(\bullet(*,U),R))
&=\Phi(U\oplus R)\,(\bl(U)+1,\,1),\\
\Phi(\circ(\bullet(*,U),R))
&=\Phi(U\oplus R)\,(1,\,1-\bl(U)),\\
\Phi(\circ(\bullet(\bullet(*,U_1),U_2),R))
&=\Phi(U_1\oplus U_2\oplus R)\,(\bl(U_1)+1,\,1-\bl(U_2)).
\end{align*}
\end{definition}

\begin{theorem}\label{thm:tree-to-fpath}
For every $n\ge 1$, the restriction $\Phi_n=\Phi|_{\mathcal A_n}$ satisfies
\[
\Phi_n(\mathcal A_n)\subseteq F_{n-1}.
\]
Moreover,
\[
\operatorname{ht}(\Phi(T))=\bl(T)-1
\qquad\text{for all }T\in\mathcal A_n.
\]
\end{theorem}

\begin{proof}
We proceed by induction on the number of leaves. The statement is obvious for $n=1$.
Assume that $T$ is non-leaf and the theorem is true for all smaller admissible
trees. We check separately the five cases.

\noindent\textbf{Case 1:} $T=\circ(*,R)$.

By induction,
\[
\operatorname{ht}(\Phi(R))=\bl(R)-1.
\]
Appending the step $(0,1)$ raises the height by $1$, hence
\[
\operatorname{ht}(\Phi(T))
=\operatorname{ht}(\Phi(R))+1
=\bl(R)
=\bl(T)-1
\]
by Lemma~\ref{lem:block-rec}. Therefore $\Phi(T)\in F_{n-1}$.

\smallskip
\noindent\textbf{Case 2:} $T=\bullet(*,R)$.

Appending $(1,1)$ does not change the height, hence
\[
\operatorname{ht}(\Phi(T))
=\operatorname{ht}(\Phi(R))
=\bl(R)-1
=\bl(T)-1.
\]
Thus $\Phi(T)\in F_{n-1}$.

\smallskip
\noindent\textbf{Case 3:} $T=\bullet(\bullet(*,U),R)$.

Set $\widehat T=U\oplus R$. By induction,
\[
\operatorname{ht}(\Phi(\widehat T))=\bl(\widehat T)-1=\bl(U)+\bl(R)-1.
\]
After appending the step $(\bl(U)+1,1)$, the new height becomes
\[
\bl(U)+\bl(R)-1 + 1-(\bl(U)+1)=\bl(R)-1=\bl(T)-1.
\]
Since $\bl(R)\ge 1$, the final height is nonnegative, so $\Phi(T)\in F_{n-1}$.

\smallskip
\noindent\textbf{Case 4:} $T=\circ(\bullet(*,U),R)$.

Set $\widehat T=U\oplus R$. Again,
\[
\operatorname{ht}(\Phi(\widehat T))=\bl(U)+\bl(R)-1.
\]
Appending $(1,1-\bl(U))$ changes the height by $-\bl(U)$, hence
\[
\operatorname{ht}(\Phi(T))
=\bl(U)+\bl(R)-1-\bl(U)
=\bl(R)-1
=\bl(T)-1.
\]
Thus $\Phi(T)\in F_{n-1}$.

\smallskip
\noindent\textbf{Case 5:}
$T=\circ(\bullet(\bullet(*,U_1),U_2),R)$.

Set $\widehat T=U_1\oplus U_2\oplus R$. Then
\[
\operatorname{ht}(\Phi(\widehat T))
=\bl(U_1)+\bl(U_2)+\bl(R)-1.
\]
Appending the step $(\bl(U_1)+1,\,1-\bl(U_2))$ changes the height by
\[
(1-\bl(U_2))-(\bl(U_1)+1)= -\bl(U_1)-\bl(U_2),
\]
and therefore
\[
\operatorname{ht}(\Phi(T))
=\bl(R)-1
=\bl(T)-1.
\]
Again $\Phi(T)\in F_{n-1}$.
\end{proof}

To prove bijectivity, we construct the inverse explicitly.

\begin{definition}
Define recursively
\[
\Psi:\bigsqcup_{m\ge 0}F_m\longrightarrow \bigsqcup_{n\ge 1}\mathcal A_n.
\]
For the empty path, $\Psi(\varnothing)=*$. Let
\[
Q=\widehat Q\,(a,b)\in F_m,
\qquad
\widehat T=\Psi(\widehat Q).
\]
Write the unique direct-sum decomposition of $\widehat T$ as
\[
\widehat T=T_1\oplus T_2\oplus\cdots\oplus T_r,
\qquad r=\bl(\widehat T)=\operatorname{ht}(\widehat Q)+1.
\]

Now define $\Psi(Q)$ according to the last step $(a,b)$:

\begin{itemize}
\item if $(a,b)=(0,1)$, then
\[
\Psi(Q)=\circ(*,\widehat T);
\]

\item if $(a,b)=(1,1)$, then
\[
\Psi(Q)=\bullet(*,\widehat T);
\]

\item if $a\ge 2$ and $b=1$, put
\[
U=T_1\oplus\cdots\oplus T_{a-1},
\qquad
R=T_a\oplus\cdots\oplus T_r,
\]
and set
\[
\Psi(Q)=\bullet(\bullet(*,U),R);
\]

\item if $a=1$ and $b\le 0$, put
\[
U=T_1\oplus\cdots\oplus T_{1-b},
\qquad
R=T_{2-b}\oplus\cdots\oplus T_r,
\]
and set
\[
\Psi(Q)=\circ(\bullet(*,U),R);
\]

\item if $a\ge 2$ and $b\le 0$, put
\[
U_1=T_1\oplus\cdots\oplus T_{a-1},\;\;\;
U_2=T_a\oplus\cdots\oplus T_{a-b},
\]
\[
R=T_{a-b+1}\oplus\cdots\oplus T_r,
\]
and set
\[
\Psi(Q)=\circ(\bullet(\bullet(*,U_1),U_2),R).
\]
\end{itemize}
\end{definition}

The construction above is legitimate. Indeed, since $Q\in F_m$, the final height of $Q$ is nonnegative. Hence, for the last step $(a,b)$, we have

\[
\operatorname{ht}(\widehat Q)+b-a\ge 0.
\]
Using
\[
r=\bl(\widehat T)=\operatorname{ht}(\widehat Q)+1,
\]
we obtain
\[
r\ge a-b+1.
\]
This inequality guarantees that the direct-sum decomposition of $\widehat T$
contains sufficiently many blocks to perform all the splits used in the definition of $\Psi$.

For example, in the third case, where $a\ge 2$ and $b=1$, the inequality gives $r\ge a$. Therefore the decomposition
\[
\widehat T=T_1\oplus T_2\oplus\cdots\oplus T_r
\]
can be split as
\[
U=T_1\oplus\cdots\oplus T_{a-1},
\qquad
R=T_a\oplus\cdots\oplus T_r.
\]
Thus the tree
\[
\Psi(Q)=\bullet(\bullet(*,U),R)
\]
is obtained from a valid block decomposition. The other cases are treated in the same way. Moreover, the resulting tree is admissible by the five root forms described in Lemma~\ref{lem:five-types}.

\begin{theorem}\label{thm:PhiPsi}
For every $n\ge 1$, the maps $\Phi_n$ and $\Psi_{n-1}$ are mutually inverse
bijections:
\[
\Phi_n:\mathcal A_n\longleftrightarrow F_{n-1}:\Psi_{n-1}.
\]
\end{theorem}

\begin{proof}
We prove  that
\[
\Psi(\Phi(T))=T
\qquad\text{and}\qquad
\Phi(\Psi(Q))=Q
\]
by induction on the number of internal vertices.

Suppose first that $T$ is given. In each of the five cases of
Lemma~\ref{lem:five-types}, the definition of $\Phi$ records
\begin{itemize}
\item the appropriate last step $(a,b)$;
\item the tree $\widehat T$ occurring before the last step;
\item the exact block lengths needed to recover $U$, $U_1$, $U_2$, and $R$.
\end{itemize}
Since the direct-sum decomposition is unique by
Lemma~\ref{lem:block-criterion}, applying $\Psi$ recovers exactly the same root
type and the same subtrees. Hence $\Psi(\Phi(T))=T$.

Conversely, let $Q=\widehat Q(a,b)\in F_m$. By induction on the length,
$\Phi(\Psi(\widehat Q))=\widehat Q$. The last step $(a,b)$ determines uniquely
which of the five root types is used in the construction of $\Psi(Q)$, and the
block decomposition of $\Psi(\widehat Q)$ determines uniquely the pieces
$U$, $U_1$, $U_2$, and $R$. Applying $\Phi$ to the resulting tree reattaches the
same last step $(a,b)$, hence $\Phi(\Psi(Q))=Q$.

For example, consider the fifth case
\[
T=\circ(\bullet(\bullet(*,U_1),U_2),R).
\]
Put
\[
\widehat T=U_1\oplus U_2\oplus R.
\]
Then, by definition of $\Phi$, we have
\[
\Phi(T)=\Phi(\widehat T)\,(\bl(U_1)+1,\,1-\bl(U_2)).
\]
Since $\widehat T$ has fewer internal vertices than $T$, by the induction
hypothesis
\[
\Psi(\Phi(\widehat T))=\widehat T.
\]
Therefore, when we apply $\Psi$ to $\Phi(T)$, first the prefix
$\Phi(\widehat T)$ reconstructs the tree
\[
\widehat T=U_1\oplus U_2\oplus R.
\]
Now the last step is
\[
(a,b)=(\bl(U_1)+1,\,1-\bl(U_2)).
\]
Hence
\[
a-1=\bl(U_1),\qquad a-b=\bl(U_1)+\bl(U_2).
\]
Thus, by the definition of $\Psi$, the direct-sum decomposition of
$\widehat T$ is split exactly into the three parts $U_1$, $U_2$, and $R$.
Consequently,
\[
\Psi(\Phi(T))
=
\circ(\bullet(\bullet(*,U_1),U_2),R)
=
T.
\]
This proves $\Psi(\Phi(T))=T$ in case \textup{(V)}.
\end{proof}

Let us summarize all the information.

\begin{table}[ht]
\centering
\scriptsize
\label{tab:cases-structure}
\setlength{\tabcolsep}{3pt}
\renewcommand{\arraystretch}{1.05}
\begin{tabular}{c|c|c|p{6.2cm}|c}
Case & Last step $(a,b)$ & Tree form & Block partition & $r_{\min}$ \\
\hline
I   & $(0,1)$
& $\circ(*,R)$
& $U=\emptyset,\;R=T_1\oplus\cdots\oplus T_r$
& $r\ge0$ \\

II  & $(1,1)$
& $\bullet(*,R)$
& $U=\emptyset,\;R=T_1\oplus\cdots\oplus T_r$
& $r\ge1$ \\

III & $(a,1),a\ge2$
& $\bullet(\bullet(*,U),R)$
& $U=T_1\oplus\cdots\oplus T_{a-1},\;R=T_a\oplus\cdots\oplus T_r$
& $r\ge a$ \\

IV  & $(1,b),b\le0$
& $\circ(\bullet(*,U),R)$
& $U=T_1\oplus\cdots\oplus T_{1-b},\;R=T_{2-b}\oplus\cdots\oplus T_r$
& $r\ge2-b$ \\

V   & $(a,b),a\ge2,b\le0$
& $\circ(\bullet(\bullet(*,U_1),U_2),R)$
& $U_1=T_1\oplus\cdots\oplus T_{a-1},\;
   U_2=T_a\oplus\cdots\oplus T_{a-b},\;
   R=T_{a-b+1}\oplus\cdots\oplus T_r$
& $r\ge a-b+1$ \\
\end{tabular}
\end{table}

\begin{table}[ht]
\centering
\scriptsize
\label{tab:cases-rules}
\setlength{\tabcolsep}{4pt}
\renewcommand{\arraystretch}{1.15}
\begin{tabular}{c|p{6.2cm}|p{6.2cm}}
Case & $\Phi$ rule & $\Psi$ rule \\
\hline
I
& $\Phi(\circ(*,R))=\Phi(R)(0,1)$
& $\Psi(\ldots(0,1))=\circ(*,\widehat T)$ \\

II
& $\Phi(\bullet(*,R))=\Phi(R)(1,1)$
& $\Psi(\ldots(1,1))=\bullet(*,\widehat T)$ \\

III
& $\Phi(\bullet(\bullet(*,U),R))=\Phi(U\oplus R)(a,1)$
& $\Psi(\widehat T(a,1))=\bullet(\bullet(*,U),R)$ \\

IV
& $\Phi(\circ(\bullet(*,U),R))=\Phi(U\oplus R)(1,b)$
& $\Psi(\widehat T(1,b))=\circ(\bullet(*,U),R)$ \\

V
& $\Phi(\circ(\bullet(\bullet(*,U_1),U_2),R))
   =\Phi(U_1\oplus U_2\oplus R)(a,b)$
& $\Psi(\widehat T(a,b))
   =\circ(\bullet(\bullet(*,U_1),U_2),R)$ \\
\end{tabular}
\end{table}

\begin{example}
    
For $n\le 3$, the bijection can be written explicitly as follows:

\[
\begin{array}{c|c|c}
\pre\text{-}\As^{\I}\langle x\rangle
& T\in\mathcal A_n 
& \Phi(T)\in F_{n-1} \\
\hline
x 
& * 
& \varnothing \\

x\succ x 
& \circ(*,*) 
& (0,1)\\

x\prec x 
& \bullet(*,*) 
& (1,1)\\

x\succ(x\succ x) 
& \circ(*,\circ(*,*)) 
& (0,1)(0,1)\\

x\prec(x\succ x) 
& \bullet(*,\circ(*,*)) 
& (0,1)(1,1)\\

x\succ(x\prec x) 
& \circ(*,\bullet(*,*)) 
& (1,1)(0,1)\\

x\prec(x\prec x) 
& \bullet(*,\bullet(*,*)) 
& (1,1)(1,1)\\

(x\prec x)\succ x 
& \circ(\bullet(*,*),*) 
& (0,1)(1,0)\\

(x\prec x)\prec x 
& \bullet(\bullet(*,*),*) 
& (0,1)(2,1)
\end{array}
\]
\end{example}


\end{document}